\documentclass{amsart}
\usepackage{amsfonts}
\usepackage{hyperref}

\newtheorem{theorem}{Theorem}
\theoremstyle{plain}

\newtheorem{corollary}{Corollary}

\newtheorem{lemma}{Lemma}

\newtheorem{property}{Property}
\newtheorem{proposition}{Proposition}
\newtheorem{remark}{Remark}

\numberwithin{equation}{section}
\input{tcilatex}

\begin{document}
\title[A new way to prove LMR]{A new way to prove L'Hospital Monotone Rules
with applications}
\author{Zhen-Hang Yang}
\address{Customer Service Center, ZPEPC Electric Power Research Institute,
Hangzhou, Zhejiang, China, 310009}
\email{yzhkm@163.com}
\date{August 10, 2014}
\thanks{This paper is in final form and no version of it will be submitted
for publication elsewhere.}

\begin{abstract}
Let $-\infty \leq a<b\leq \infty $. Let $f$ and $g$ be differentiable
functions on $(a,b)$ and let $g^{\prime }\neq 0$ on $(a,b)$. By introducing
an auxiliary function $H_{f,g}:=\left( f^{\prime }/g^{\prime }\right) g-f$,
we easily prove L'Hoipital rules for monotonicity. This offer a natural and
concise way so that those rules are easier to be understood. Using our
L'Hospital Piecewise Monotone Rules (for short, LPMR), we establish three
new sharp inequalities for hyperbolic and trigonometric functions as well as
bivariate means, which supplement certain known results.
\end{abstract}

\maketitle

\section{Introduction}

As pointed out by Anderson et al in \cite{Anderson-AMM-113(9)-2006} "if one
is attempting to establish the monotonicity of a quotient of two functions,
one often finds that the derivative of the quotient is quite messy and the
process tedious". This inspired mathematicians to find the refinements for
proving monotonicity of quotients. In 1982, the first such refinement was
presented in \cite{Cheeger-JDG-17-1982} by Cheegeret al. which states that

\begin{theorem}
\label{T-Cheeger}If $f$ and $g$ are positive integrable functions on $%
\mathbb{R}$ such that $f/g$ is decreasing, then the function%
\begin{equation*}
x\mapsto \frac{\int_{0}^{x}f(t)dt}{\int_{0}^{x}g(t)dt}
\end{equation*}%
is decreasing.
\end{theorem}

In 1993, Anderson et al \cite{Anderson-PJM-160(1)-1993} (also see \cite[%
Theorem 1.25]{Anderson-CIIQM-1997}) developed a rule that applies to a wide
class of quotients of functions by using the Cauchy mean value theorem.

\begin{theorem}[{\protect\cite[Lemma 2.2]{Anderson-PJM-160(1)-1993}}]
\label{T-Anderson}For $-\infty <a<b<\infty $ let $f,g:[a,b]\rightarrow 
\mathbb{R}$ be differentiable functions such that $g^{\prime }(x)\neq 0$ for
each $x\in (a,b)$. If $f^{\prime }/g^{\prime }$ is increasing (decreasing)
on $(a,b)$ then so is $x\mapsto \left( f\left( x\right) -f\left( a\right)
\right) /\left( g\left( x\right) -g\left( a\right) \right) $.
\end{theorem}

In \cite{Vamanamurthy-JMAA-183(1)-1994}, Vamanamurthy and Vuorinen proved
that this theorem is also true for the function $x\mapsto \left( f\left(
x\right) -f\left( b\right) \right) /\left( g\left( x\right) -g\left(
b\right) \right) $. The result together with Theorem \ref{T-Anderson} can be
unified as

\begin{theorem}[{\protect\cite[Theorem 2]{Anderson-AMM-113(9)-2006}}]
\label{T-Anderson2}Let $-\infty <a<b<\infty $, and let $f,g:[a,b]\rightarrow 
\mathbb{R}$ be continuous functions that are differentiable on $\left(
a,b\right) $, with $f\left( a\right) =g\left( a\right) =0$ or $f\left(
b\right) =g\left( b\right) =0$. Assume that $g^{\prime }(x)\neq 0$ for each $%
x$ in $(a,b)$. If $f^{\prime }/g^{\prime }$ is increasing (decreasing) on $%
(a,b)$ then so is $f/g$.
\end{theorem}

Because of the similarity of the hypotheses to those of L'Hospital's rules,
Anderson et al. refers to this result as the "L'Hospital Monotone Rule" (or,
for short, LMR).

Theorem \ref{T-Anderson2} assumes that $a$ and $b$ are finite, but the rule
is true in the case when $a$ or $b$ is infinite. This was first proved by
Pinelis in \cite[Proposition 1.1]{Pinelis-JIPAM-3(1)-2002-5}, and known as
"special-case l'Hospital-type rule for monotonicity" \cite[Proposition 1.2,
Corollary 1.3]{Pinelis-JIPAM-2(3)-2001}.

\begin{theorem}
\label{T-Pinelis1}Let $-\infty \leq a<b\leq \infty $. Let $f$ and $g$ be
differentiable functions on the interval $(a,b)$. Assume also that the
derivative $g^{\prime }$ is nonzero and does not change sign on $(a,b)$.
Suppose that $f(a^{+})=g(a^{+})=0$ or $f(b^{-})=g(b^{-})=0$. If $f^{\prime
}/g^{\prime }$ is increasing (decreasing) on $(a,b)$ then so is $f/g$.
\end{theorem}

For the case when $f$ and $g$ are probability tail functions, a proof of
Theorem \ref{T-Pinelis1} may be found in \cite{Pinelis-AS-22(1)-1994}. In 
\cite[Proposition 1.2, Corollary 1.3]{Pinelis-JIPAM-2(3)-2001}, a more
general monotonicity rule was proved, which does not require that $f$ and $g$
vanish at an endpoint of the interval.

\begin{theorem}
\label{T-Pinelis2}Let $-\infty \leq a<b\leq \infty $. Let $f$ and $g$ be
differentiable functions on the interval $(a,b)$. Assume also that the
derivative $g^{\prime }$ is nonzero and does not change sign on $(a,b)$.

(i) If $gg^{\prime }>0$ on $\left( a,b\right) $, $\lim_{c\downarrow a}\sup
g^{2}\left( c\right) \left( f/g\right) ^{\prime }\left( c\right) /|g^{\prime
}\left( c\right) |\geq 0$, and $f^{\prime }/g^{\prime }$ is increasing on $%
(a,b)$, then $\left( f/g\right) ^{\prime }>0$ on $(a,b)$.

(ii) If $gg^{\prime }>0$ on $\left( a,b\right) $, $\lim_{c\downarrow a}\inf
g^{2}\left( c\right) \left( f/g\right) ^{\prime }\left( c\right) /|g^{\prime
}\left( c\right) |\leq 0$, and $f^{\prime }/g^{\prime }$ is decreasing on $%
(a,b)$, then $\left( f/g\right) ^{\prime }>0$ on $(a,b)$.

(iii) If $gg^{\prime }<0$ on $\left( a,b\right) $, $\lim_{c\uparrow b}\inf
g^{2}\left( c\right) \left( f/g\right) ^{\prime }\left( c\right) /|g^{\prime
}\left( c\right) |\leq 0$, and $f^{\prime }/g^{\prime }$ is decreasing on $%
(a,b)$, then $\left( f/g\right) ^{\prime }<0$ on $(a,b)$.

(iv) If $gg^{\prime }<0$ on $\left( a,b\right) $, $\lim_{c\uparrow b}\sup
g^{2}\left( c\right) \left( f/g\right) ^{\prime }\left( c\right) /|g^{\prime
}\left( c\right) |\geq 0$, and $f^{\prime }/g^{\prime }$ is increasing on $%
(a,b)$, then $\left( f/g\right) ^{\prime }>0$ on $(a,b)$.
\end{theorem}

Almost at the same moment, Pinelis in \cite[Lemma 2.1]%
{Pinelis-JIPAM-7(2)-2006} and Anderson et al. in \cite[Theorem 5]%
{Anderson-AMM-113(9)-2006} showed independently another useful monotonicity
result. Here we quote Pinelis's result as follows. For convenience in what
follows we sometime use the notations "$\nearrow $" and "$\searrow $" to
denote "increasing" and "decreasing", respectively; while "$\nearrow
\searrow $" ("$\searrow \nearrow $") means that there is a $c\in (a,b)$ such
that a given function is increasing (decreasing) on $\left( a,c\right) $ and
decreasing (increasing) on $\left( c,b\right) $.

\begin{proposition}[{\protect\cite[Theorem 5]{Anderson-AMM-113(9)-2006}, 
\protect\cite[Lemma 2.1]{Pinelis-JIPAM-7(2)-2006}}]
\label{P-P-A}Let $-\infty \leq a<b\leq \infty $. Let $f$ and $g$ be
differentiable functions on the interval $(a,b)$. Assume also that the
derivative $g^{\prime }$ is nonzero and does not change sign on $(a,b)$.
Then the monotonicity pattern ($\nearrow $ or $\searrow $) of the function $%
\tilde{\rho}=g^{2}\left( f/g\right) ^{\prime }/|g^{\prime }|$ on $(a,b)$ is
determined by the monotonicity pattern of $\rho =f^{\prime }/g^{\prime }$
and the sign of $gg^{\prime }$, according to the following table%
\begin{equation*}
\begin{array}{c}
\begin{tabular}{|c|c|c|}
\hline
$\rho =f^{\prime }/g^{\prime }$ & $gg^{\prime }$ & $\tilde{\rho}=g^{2}\left(
f/g\right) ^{\prime }/|g^{\prime }|$ \\ \hline
$\nearrow $ & $>0$ & $\nearrow $ \\ \hline
$\searrow $ & $>0$ & $\searrow $ \\ \hline
$\nearrow $ & $<0$ & $\searrow $ \\ \hline
$\searrow $ & $<0$ & $\nearrow $ \\ \hline
\end{tabular}
\\ 
\\ 
\text{Table 1}%
\end{array}%
\end{equation*}
\end{proposition}

Based on Proposition \ref{P-P-A}, Pinelis easily deduced "refined general
rules for monotonicity" including three assertions in \cite[Corollary 3.1,
3.2, 3.3]{Pinelis-JIPAM-7(2)-2006} etc. So the proposition is called "Key
Lemma" by Pinelis.

In the same paper, Pinelis also proved the derived general rules under the
same special condition that $f(a^{+})=g(a^{+})=0$ or $f(b^{-})=g(b^{-})=0$,
that is, the following assertion.

\begin{theorem}[{\protect\cite[Propositon 4.4]{Pinelis-JIPAM-7(2)-2006}}]
\label{T-Pinelis4}The conditions are the same as Proposition \ref{P-P-A}'s.
If $f(a^{+})=g(a^{+})=0$ or $f(b^{-})=g(b^{-})=0$, then the monotonicity
rules given by the following table are true.%
\begin{equation*}
\begin{array}{c}
\begin{tabular}{|c|c|c|c|c|}
\hline
endpoint condition & $\rho =f^{\prime }/g^{\prime }$ & $\tilde{\rho}\left(
a^{+}\right) $ & $\tilde{\rho}\left( b^{-}\right) $ & $r=f/g$ \\ \hline
$f(a^{+})=g(a^{+})=0$ & $\nearrow \searrow $ &  & $\geq 0$ & $\nearrow $ \\ 
\hline
$f(a^{+})=g(a^{+})=0$ & $\nearrow \searrow $ &  & $<0$ & $\nearrow \searrow $
\\ \hline
$f(a^{+})=g(a^{+})=0$ & $\searrow \nearrow $ &  & $\leq 0$ & $\searrow $ \\ 
\hline
$f(a^{+})=g(a^{+})=0$ & $\searrow \nearrow $ &  & $>0$ & $\searrow \nearrow $
\\ \hline
$f(b^{-})=g(b^{-})=0$ & $\nearrow \searrow $ & $\leq 0$ &  & $\searrow $ \\ 
\hline
$f(b^{-})=g(b^{-})=0$ & $\nearrow \searrow $ & $>0$ &  & $\nearrow \searrow $
\\ \hline
$f(b^{-})=g(b^{-})=0$ & $\searrow \nearrow $ & $\geq 0$ &  & $\nearrow $ \\ 
\hline
$f(b^{-})=g(b^{-})=0$ & $\searrow \nearrow $ & $<0$ &  & $\searrow \nearrow $
\\ \hline
\end{tabular}
\\ 
\\ 
\text{Table 2}%
\end{array}%
\end{equation*}
\end{theorem}

L'Hospital-type rules for monotonicity for "Non-strict" and in the discrete
case can see \cite{Pinelis-JIPAM-8(1)-2007}, \cite{Pinelis-MIA-11(4)-2008}.

The foregoing rules have been widely used very effectively in the study of
some areas, such as elliptic integrals and other special functions in
quasiconformal theory quasiconformal mappings \cite{Anderson-PJM-160(1)-1993}%
, \cite{Anderson-CIIQM-1997}, \cite{Anderson-PA-2001}, \cite%
{Anderson-PJM-192-2000}, \cite{Alzer-JCAM-172-2004}, \cite%
{Chu-AAA-2011-697547}, \cite{Song-JMI-7(4)-2013}, approximation theory \cite%
{Pinelis-JIPAM-3(2)-2002-20}, differential geometry \cite%
{Cheeger-JDG-17-1982}, \cite{Chavel-RG-CUPC-1993}, \cite%
{Gromov-LNM-1200-1986}, \cite{Pinelis-JIPAM-6(4)-2005}, information theory 
\cite{Pinelis-JIPAM-3(1)-2002-5}, \cite{Pinelis-JIPAM-2(3)-2001}, statistics
and probability \cite{Pinelis-AS-22(1)-1994}, \cite{Pinelis-JIPAM-2(3)-2001}%
, \cite{Pinelis-JIPAM-3(1)-2002-5}, \cite{Pinelis-JIPAM-3(2)-2002-20}, \cite%
{Pinelis-JIPAM-3(1)-2002-7}, analytic inequalities \cite%
{Anderson-AMM-113(9)-2006}, \cite{Pinelis-JIPAM-2(3)-2001}, \cite%
{Pinelis-JIPAM-3(1)-2002-5}, \cite{Pinelis-AMM-111-2004}, \cite%
{Zhu-AML-19-2006}, \cite{Zhu-JIA-2007-67430}, \cite{Zhu-CMA-58-2009}, \cite%
{Zhu-AAA-2009-485842}, \cite{Zhu-JIA-2012-303}, \cite{Yang-JMI-7(4)-2013}, 
\cite{Yang-AAA-2014-601839}, etc.

\begin{remark}
The foregoing rules can be divided into two categories: The first one
require the endpoint condition that $f(a^{+})=g(a^{+})=0$ or $%
f(b^{-})=g(b^{-})=0$; the second one dose not require. For unity, we
suggestion that:

(i) Adopting Anderson et al's term, the first category of rules is
collectively called "L'Hospital Monotone Rules", for short, LMR. If the
monotonicity of $f^{\prime }/g^{\prime }$ on $\left( a,b\right) $ has the
pattern that "$\nearrow \searrow $" or "$\searrow \nearrow $", then it is
also called "L'Hospital Piecewise Monotone Rules", for short, LPMR.

(ii) According to Pinelis' term, the second category of rules is
collectively named "General L'Hospital Monotone Rules", for short, GLMR.
\end{remark}

We note that the "Key Lemma" is indeed crucial, as mentioned in \cite%
{Pinelis-JIPAM-7(2)-2006}, however, its two proofs offered by Pinelis in 
\cite{Pinelis-JIPAM-7(2)-2006}, "one proof is short and self-contained, even
if somewhat cryptic; the other proof is longer but apparently more
intuitive".

Because of this, it is meaningful to find a natural and concise way to
expound or prove those L'Hospital rules for monotonicity so that they are
better understood. This is the first aim of this paper. For this end, in
Section 2 we introduce an auxiliary function $H_{f,g}$ that has some simple
and useful properties. By virtue of this auxiliary function we easily
restated or prove "General L'Hospital (Piecewise) Monotone Rules".

Our second aim is to apply "L'Hospital Piecewise Monotone Rules" to
establish some new sharp inequalities for hyperbolic and trigonometric
functions as well as bivariate means. These are supplements to known ones.

\section{A natural and concise way to prove monotone rules}

In this section, we begin with introducing an important auxiliary function.

Let $-\infty \leq a<b\leq \infty $. Let $f$ and $g$ be differentiable
functions on $(a,b)$. Assume also that $g^{\prime }\neq 0$ on $(a,b)$. Then
we define%
\begin{equation}
H_{f,g}:=\frac{f^{\prime }}{g^{\prime }}g-f.  \label{H_f,g}
\end{equation}%
For latter use, we present some properties of this auxiliary function. The
following property is well but simple, so we omit the process of proof.

\begin{property}
\label{P of H1}Let $-\infty \leq a<b\leq \infty $. Let $f$ and $g$ be
differentiable functions on $(a,b)$ and let $g^{\prime }\neq 0$ on $(a,b)$.
Let $H_{f,g}$ be defined on $\left( a,b\right) $ by (\ref{H_f,g}). Then the
following statements are true.

(i) $H_{f,g}$ is even with respect to $g$ and odd with respect to $f$, that
is,%
\begin{equation}
H_{f,g}\left( x\right) =H_{f,-g}\left( x\right) =-H_{-f,g}\left( x\right)
=-H_{-f,-g}\left( x\right) .  \label{H-sr}
\end{equation}

(ii) If $f$ and $g$ are twice differentiable on $(a,b)$, then%
\begin{equation}
H_{f,g}^{\prime }=\left( \frac{f^{\prime }}{g^{\prime }}\right) ^{\prime }g,
\label{dH_f,g}
\end{equation}

(iii) If $g\neq 0$ on $\left( a,b\right) $, then%
\begin{equation}
\left( \frac{f}{g}\right) ^{\prime }=\frac{g^{\prime }}{g^{2}}H_{f,g},
\label{df/g}
\end{equation}%
and therefore,%
\begin{equation}
\func{sgn}\left( \frac{f}{g}\right) ^{\prime }=\func{sgn}g^{\prime }\func{sgn%
}H_{f,g}.  \label{sgnd(f/g)}
\end{equation}
\end{property}

\begin{remark}
By (\ref{df/g}), we have $H_{f,g}=g^{2}\left( f/g\right) ^{\prime
}/g^{\prime }$. It is seen that our auxiliary function $H_{f,g}$ is similar
to Pinelis's function $\tilde{\rho}=g^{2}\left( f/g\right) ^{\prime
}/|g^{\prime }|$, to be exact,%
\begin{equation*}
\tilde{\rho}=H_{f,g}\text{ if }g^{\prime }>0\text{ on }\left( a,b\right) 
\text{ and }\tilde{\rho}=-H_{f,g}\text{ if }g^{\prime }<0\text{ on }\left(
a,b\right) .
\end{equation*}
\end{remark}

The following property of $H_{f,g}$ is very useful in the sequel, crucial
role played by which is similar to the "Key Lemma" presented by Pinelis. But
our proof is natural and simple and so it is easy to be understood.

\begin{property}
\label{P of H2}Let $-\infty \leq a<b\leq \infty $. Let $f$ and $g$ be
differentiable functions on $(a,b)$ and let $g^{\prime }\neq 0$ on $(a,b)$.

(i) If $g>0$ on $\left( a,b\right) $, then $H_{f,g}$ defined on $\left(
a,b\right) $ by (\ref{H_f,g}) is increasing (decreasing) according as $%
f^{\prime }/g^{\prime }$ is increasing (decreasing).

(ii) If $g<0$ on $(a,b)$, then $H_{f,g}$ is decreasing (increasing)
according as $f^{\prime }/g^{\prime }$ is increasing (decreasing).
\end{property}

\begin{proof}
From those symmetry relations (\ref{H-sr}) of $H_{f,g}$, it suffices to
prove that $H_{f,g}$ is increasing on $\left( a,b\right) $ in the case when $%
g>0$ and $f^{\prime }/g^{\prime }$ is increasing on $\left( a,b\right) $.

Let $x_{1}$, $x_{2}\in \left( a,b\right) $ with $x_{1}<x_{2}$. We
distinguish two cases to prove $H_{f,g}\left( x_{1}\right) <H_{f,g}\left(
x_{2}\right) $, that is,%
\begin{equation}
H_{f,g}\left( x_{1}\right) -H_{f,g}\left( x_{2}\right) =\frac{f^{\prime
}\left( x_{1}\right) }{g^{\prime }\left( x_{1}\right) }g\left( x_{1}\right) -%
\frac{f^{\prime }\left( x_{2}\right) }{g^{\prime }\left( x_{2}\right) }%
g\left( x_{2}\right) -\left( f\left( x_{1}\right) -f\left( x_{2}\right)
\right) <0.  \label{Hx1-Hx2}
\end{equation}

Case 1: $g,g^{\prime }>0$ on $\left( a,b\right) $ and $f^{\prime }/g^{\prime
}$ is increasing on $\left( a,b\right) $. Applying $f^{\prime }\left(
x_{1}\right) /g^{\prime }\left( x_{1}\right) <f^{\prime }\left( x_{2}\right)
/g^{\prime }\left( x_{2}\right) $ to the first term in the middle of (\ref%
{Hx1-Hx2}) and next making simple identical transformations yield%
\begin{eqnarray*}
H_{f,g}\left( x_{1}\right) -H_{f,g}\left( x_{2}\right) &<&\frac{f^{\prime
}\left( x_{2}\right) }{g^{\prime }\left( x_{2}\right) }g\left( x_{1}\right) -%
\frac{f^{\prime }\left( x_{2}\right) }{g^{\prime }\left( x_{2}\right) }%
g\left( x_{2}\right) -\left( f\left( x_{1}\right) -f\left( x_{2}\right)
\right) \\
&=&\left( g\left( x_{1}\right) -g\left( x_{2}\right) \right) \left( \frac{%
f^{\prime }\left( x_{2}\right) }{g^{\prime }\left( x_{2}\right) }-\frac{%
f\left( x_{1}\right) -f\left( x_{2}\right) }{g\left( x_{1}\right) -g\left(
x_{2}\right) }\right) <0,
\end{eqnarray*}%
where $\left( g\left( x_{1}\right) -g\left( x_{2}\right) \right) <0$ due to $%
g^{\prime }>0$ on $\left( a,b\right) $, while the increasing property of $%
f^{\prime }/g^{\prime }$ on $\left( a,b\right) $ implies that there exits a $%
\theta \in \left( 0,1\right) $ such that%
\begin{equation*}
\frac{f^{\prime }\left( x_{2}\right) }{g^{\prime }\left( x_{2}\right) }-%
\frac{f\left( x_{1}\right) -f\left( x_{2}\right) }{g\left( x_{1}\right)
-g\left( x_{2}\right) }=\frac{f^{\prime }\left( x_{2}\right) }{g^{\prime
}\left( x_{2}\right) }-\frac{f^{\prime }\left( x_{1}+\theta \left(
x_{2}-x_{1}\right) \right) }{g^{\prime }\left( x_{1}+\theta \left(
x_{2}-x_{1}\right) \right) }>0.
\end{equation*}

Case 2: $g>0,g^{\prime }<0$ on $\left( a,b\right) $ and $f^{\prime
}/g^{\prime }$ is increasing on $\left( a,b\right) $. Similar to Case 1,
using $f^{\prime }\left( x_{2}\right) /g^{\prime }\left( x_{2}\right)
>f^{\prime }\left( x_{1}\right) /g^{\prime }\left( x_{1}\right) $ to the
second term in the middle of (\ref{Hx1-Hx2}) and next making simple
identical transformations give%
\begin{eqnarray*}
H_{f,g}\left( x_{1}\right) -H_{f,g}\left( x_{2}\right) &<&\frac{f^{\prime
}\left( x_{1}\right) }{g^{\prime }\left( x_{1}\right) }g\left( x_{1}\right) -%
\frac{f^{\prime }\left( x_{1}\right) }{g^{\prime }\left( x_{1}\right) }%
g\left( x_{2}\right) -\left( f\left( x_{1}\right) -f\left( x_{2}\right)
\right) \\
&=&\left( g\left( x_{1}\right) -g\left( x_{2}\right) \right) \left( \frac{%
f^{\prime }\left( x_{1}\right) }{g^{\prime }\left( x_{1}\right) }-\frac{%
f\left( x_{1}\right) -f\left( x_{2}\right) }{g\left( x_{1}\right) -g\left(
x_{2}\right) }\right) <0,
\end{eqnarray*}%
where $\left( g\left( x_{1}\right) -g\left( x_{2}\right) \right) >0$ due to $%
g^{\prime }>0$ on $\left( a,b\right) $, while the increasing property of $%
f^{\prime }/g^{\prime }$ on $\left( a,b\right) $ reveals that there exits a $%
\theta \in \left( 0,1\right) $ such that%
\begin{equation*}
\frac{f^{\prime }\left( x_{1}\right) }{g^{\prime }\left( x_{1}\right) }-%
\frac{f\left( x_{1}\right) -f\left( x_{2}\right) }{g\left( x_{1}\right)
-g\left( x_{2}\right) }=\frac{f^{\prime }\left( x_{1}\right) }{g^{\prime
}\left( x_{1}\right) }-\frac{f^{\prime }\left( x_{1}+\theta \left(
x_{2}-x_{1}\right) \right) }{g^{\prime }\left( x_{1}+\theta \left(
x_{2}-x_{1}\right) \right) }<0.
\end{equation*}

By the symmetry relations \ref{H-sr}, we clearly see that:

If $g>0,f^{\prime }/g^{\prime }\searrow $, then $\left( -f\right) ^{\prime
}/g^{\prime }\nearrow $, and then by $H_{f,g}=-H_{-f,g}$ it is seen that $%
H_{f,g}\searrow $.

If $g<0,f^{\prime }/g^{\prime }\nearrow $, then $-g>0$, $\left( -f\right)
^{\prime }/\left( -g\right) ^{\prime }\nearrow $, and then by $H_{f,g}\left(
x\right) =-H_{-f,-g}\left( x\right) $ it is seen that $H_{f,g}\searrow $.

If $g<0,f^{\prime }/g^{\prime }\searrow $, then $-g>0$, $f^{\prime }/\left(
-g\right) ^{\prime }\nearrow $, and then by $H_{f,g}\left( x\right)
=H_{f,-g}\left( x\right) $ it is seen that $H_{f,g}\nearrow .$

This completes the proof.
\end{proof}

\begin{remark}
If $f$ and $g$ are twice differentiable functions on $(a,b)$, then by (ii)
of Property \ref{P of H1} we have 
\begin{equation*}
\func{sgn}H_{f,g}^{\prime }=\func{sgn}\left( \frac{f^{\prime }}{g^{\prime }}%
\right) ^{\prime }\func{sgn}g
\end{equation*}
Thus the proof of Property \ref{P of H2} will be more simple.
\end{remark}

Utilizing Property \ref{P of H2}, we can easily prove our "General
L'Hospital Monotone Rule" that dose not require the endpoint conditions. Due
to the simple relation $f/\left( -g\right) =-\left( f/g\right) $, for
clarity in statement, we assume that $g\left( x\right) >0$ on $\left(
a,b\right) $ in the following theorem.

\begin{theorem}[GLMR]
\label{T-Yang1}Let $-\infty \leq a<b\leq \infty $. Let $f$ and $g$ be
differentiable functions on $(a,b)$ and let $g^{\prime }\neq 0$ and $g>0$ on 
$(a,b)$.

(i) If $g^{\prime }>\left( <\right) 0$ and $f^{\prime }/g^{\prime }$ is
monotone on $\left( a,b\right) $, then $\left( f/g\right) ^{\prime }>\left(
<\right) 0$ on $\left( a,b\right) $ if and only if $\min \left(
H_{f,g}\left( a^{+}\right) ,H_{f,g}\left( b^{-}\right) \right) \geq 0$.

(ii) If $g^{\prime }>\left( <\right) 0$ and $f^{\prime }/g^{\prime }$ is
monotone on $\left( a,b\right) $, then $\left( f/g\right) ^{\prime }<\left(
>\right) 0$ on $\left( a,b\right) $ if and only if $\max \left(
H_{f,g}\left( a^{+}\right) ,H_{f,g}\left( b^{-}\right) \right) \leq 0$.

(iii) If $g^{\prime }>\left( <\right) 0$ and $f^{\prime }/g^{\prime }$ is
increasing on $\left( a,b\right) $, and $H_{f,g}\left( a^{+}\right) <0$ and $%
H_{f,g}\left( b^{-}\right) >0$, then there is a unique $x_{0}\in \left(
a,b\right) $ such that $\left( f/g\right) ^{\prime }<\left( >\right) 0$ on $%
\left( a,x_{0}\right) $ and $\left( f/g\right) ^{\prime }>\left( <\right) 0$
on $\left( x_{0},b\right) $.

(iv) If $g^{\prime }>\left( <\right) 0$ and $f^{\prime }/g^{\prime }$ is
decreasing on $\left( a,b\right) $, and $H_{f,g}\left( a^{+}\right) >0$ and $%
H_{f,g}\left( b^{-}\right) <0$, then there is a unique $x_{0}\in \left(
a,b\right) $ such that $\left( f/g\right) ^{\prime }>\left( <\right) 0$ on $%
\left( a,x_{0}\right) $ and $\left( f/g\right) ^{\prime }<\left( >\right) 0$
on $\left( x_{0},b\right) $.
\end{theorem}

\begin{proof}
Due to (\ref{sgnd(f/g)}), that is, $\func{sgn}\left( f/g\right) ^{\prime }=%
\func{sgn}g^{\prime }\func{sgn}H_{f,g}$, it suffice to deal with $\func{sgn}%
H_{f,g}\left( x\right) $. By part (i) of Theorem \ref{T-Yang1}, when $g>0$
on $\left( a,b\right) $, if $f^{\prime }/g^{\prime }$ is increasing
(decreasing) on $\left( a,b\right) $, then so is $H_{f,g}$. It follows that 
\begin{equation*}
\func{sgn}H_{f,g}\left( x\right) =\left\{ 
\begin{array}{cc}
\geq 0 & \text{if and only if }\min \left( H_{f,g}\left( a^{+}\right)
,H_{f,g}\left( b^{-}\right) \right) \geq 0, \\ 
\leq 0 & \text{if and only if }\max \left( H_{f,g}\left( a^{+}\right)
,H_{f,g}\left( b^{-}\right) \right) \leq 0,%
\end{array}%
\right.
\end{equation*}%
which implies (i) and (ii) in this proposition.

If $H_{f,g}\left( a^{+}\right) <\left( >\right) 0$ and $H_{f,g}\left(
b^{-}\right) >\left( <\right) 0$, then there is a unique $x_{0}\in \left(
a,b\right) $ such that $H_{f,g}\left( x\right) <\left( >\right) 0$ on $%
\left( a,x_{0}\right) $ and $H_{f,g}\left( x\right) >\left( <\right) 0$ on $%
\left( x_{0},b\right) $. This together with (\ref{sgnd(f/g)}) proves (iii)
and (iv) in this proposition.
\end{proof}

\begin{remark}
\label{Remark H(a+)=H(b-)=0}In Theorem \ref{T-Yang1}, the endpoint condition
that $f(a^{+})=g(a^{+})=0$ implies that $H_{f,g}\left( a^{+}\right) =0$,
because of%
\begin{eqnarray*}
H_{f,g}\left( a^{+}\right) &=&\lim_{x\rightarrow a^{+}}\left( \frac{%
f^{\prime }\left( x\right) }{g^{\prime }\left( x\right) }g\left( x\right)
-f\left( x\right) \right) =\lim_{x\rightarrow a^{+}}\left( \left( \frac{%
f^{\prime }\left( x\right) }{g^{\prime }\left( x\right) }-\frac{f\left(
x\right) }{g\left( x\right) }\right) g\left( x\right) \right) \\
&=&\lim_{x\rightarrow a^{+}}\left( \frac{f^{\prime }\left( x\right) }{%
g^{\prime }\left( x\right) }-\frac{f\left( x\right) }{g\left( x\right) }%
\right) \lim_{x\rightarrow a^{+}}g\left( x\right) =0,
\end{eqnarray*}%
where we have used the L'Hospital rule for indeterminate form $0/0$.
Similarly, if $f(b^{-})=g(b^{-})=0$ then there must be $H_{f,g}\left(
b^{-}\right) =0$.
\end{remark}

\begin{remark}
When $g^{\prime }\left( x\right) >0$ on $\left( a,b\right) $ and $%
f(a^{+})=g(a^{+})=0$, we see that $H_{f,g}\left( a^{+}\right) =0$ and $%
g\left( x\right) >g(a^{+})=0$. By Property \ref{P of H2}, if $f^{\prime
}/g^{\prime }$ is increasing on $\left( a,b\right) $, then so is $H_{f,g}$,
and therefore, $\min \left( H_{f,g}\left( a^{+}\right) ,H_{f,g}\left(
b^{-}\right) \right) =0$. Similarly, if $f(b^{-})=g(b^{-})=0$, then $%
H_{f,g}\left( b^{-}\right) =0$ and $g\left( x\right) >g(b^{-})=0$ according
as $g^{\prime }\left( x\right) <0$ on $\left( a,b\right) $. If $f^{\prime
}/g^{\prime }$ is increasing on $\left( a,b\right) $, then so is $H_{f,g}$,
and hence, $\max \left( H_{f,g}\left( a^{+}\right) ,H_{f,g}\left(
b^{-}\right) \right) =0$.

Employing Property \ref{P of H2}, we can obtain Theorem \ref{T-Pinelis1} in
the case when $g\left( x\right) >0$ on $\left( a,b\right) $. Then by the
relation $f/\left( -g\right) =-\left( f/g\right) $, we can infer Theorem \ref%
{T-Pinelis1} in the case when $g\left( x\right) <0$ on $\left( a,b\right) $.
\end{remark}

Now we state and prove our "L'Hospital Piecewise Monotone Rule" similar to
Theorem \ref{T-Pinelis4}. For own convenience, we distinguish two cases when 
$f(a^{+})=g(a^{+})=0$ and when $f(b^{-})=g(b^{-})=0$ to state and prove them.

\begin{theorem}[LPMR]
\label{T-Pinelis4-Ya}Let $-\infty \leq a<b\leq \infty $ and let $H_{f,g}$ be
defined by (\ref{H_f,g}). Suppose that (i) $f$ and $g$ are differentiable
functions on $(a,b)$; (ii) $g^{\prime }\neq 0$ on $(a,b)$; (ii) $%
f(a^{+})=g(a^{+})=0$; (iv) there is a $c\in \left( a,b\right) $ such that $%
f^{\prime }/g^{\prime }$ is increasing (decreasing) on $\left( a,c\right) $
and decreasing (increasing) on $\left( c,b\right) $. Then

(i) when $\func{sgn}g^{\prime }\func{sgn}H_{f,g}\left( b^{-}\right) \geq
\left( \leq \right) 0$, $f/g$ is increasing (decreasing) on $\left(
a,b\right) $;

(ii) when $\func{sgn}g^{\prime }\func{sgn}H_{f,g}\left( b^{-}\right) <\left(
>\right) 0$, there is a unique number $x_{a}\in \left( a,b\right) $ such
that $f/g$ is increasing (decreasing) on $\left( a,x_{a}\right) $ and
decreasing (increasing) on $\left( x_{a},b\right) $.
\end{theorem}

\begin{proof}
Firstly, the assumption that $g^{\prime }\neq 0$ on $(a,b)$ together with
that $f(a^{+})=g(a^{+})=0$ implies that $g\left( x\right) >\left( <\right)
g\left( a^{+}\right) =0$ if $g^{\prime }>\left( <\right) 0$ on $(a,b)$.

Secondly, the endpoint condition that $f(a^{+})=g(a^{+})=0$ yields $%
H_{f,g}\left( a^{+}\right) =0$ due to Remark \ref{Remark H(a+)=H(b-)=0}.

Thirdly, by Theorem \ref{T-Yang1}, if $g\left( x\right) >0$ on $\left(
a,b\right) $, then $H_{f,g}$ has the same monotonicity with $f^{\prime
}/g^{\prime }$ on $\left( a,b\right) $; while $g\left( x\right) <0$ on $%
\left( a,b\right) $, their monotonicity are reversed.

Fourthly, by the monotonicity of $H_{f,g}$ on $\left( a,b\right) $ together
with the signs of $H_{f,g}\left( a^{+}\right) $ and $H_{f,g}\left(
b^{-}\right) $, we can determine the sign of $H_{f,g}\left( x\right) $, that
is, $\func{sgn}\left( H_{f,g}\right) $.

Lastly, using the formula (\ref{sgnd(f/g)}), that is, $\func{sgn}\left(
f/g\right) ^{\prime }=\func{sgn}g^{\prime }\func{sgn}H_{f,g}$, we deduce the
sign of $\left( f/g\right) ^{\prime }$.

Further details can see Table 3, where "$\left( +,-\right) $" or "$\left(
-,+\right) $" means that there is a number $c\in \left( a,b\right) $ such
that a given function is positive (negative) on $\left( a,c\right) $ and
negative (positive) on $\left( c,b\right) $ (similarly Table 4).%
\begin{equation*}
\begin{array}{l}
\begin{tabular}{|c|c|c|c|c|c|c|c|}
\hline
$f^{\prime }/g^{\prime }$ & $g^{\prime }$ & $g$ & $H_{f,g}$ & $H_{f,g}\left(
a^{+}\right) $ & $H_{f,g}\left( b^{-}\right) $ & $\func{sgn}\left(
H_{f,g}\right) $ & $\func{sgn}\left( f/g\right) ^{\prime }$ \\ \hline
$\nearrow \searrow $ & $+$ & $+$ & $\nearrow \searrow $ & $0$ & $\geq 0$ & $%
+ $ & $+$ \\ \hline
$\nearrow \searrow $ & $+$ & $+$ & $\nearrow \searrow $ & $0$ & $<0$ & $%
\left( +,-\right) $ & $\left( +,-\right) $ \\ \hline
$\searrow \nearrow $ & $+$ & $+$ & $\searrow \nearrow $ & $0$ & $\leq 0$ & $%
- $ & $-$ \\ \hline
$\searrow \nearrow $ & $+$ & $+$ & $\searrow \nearrow $ & $0$ & $>0$ & $%
\left( -,+\right) $ & $\left( -,+\right) $ \\ \hline
$\nearrow \searrow $ & $-$ & $-$ & $\searrow \nearrow $ & $0$ & $\leq 0$ & $%
- $ & $+$ \\ \hline
$\nearrow \searrow $ & $-$ & $-$ & $\searrow \nearrow $ & $0$ & $>0$ & $%
\left( -,+\right) $ & $\left( +,-\right) $ \\ \hline
$\searrow \nearrow $ & $-$ & $-$ & $\nearrow \searrow $ & $0$ & $\geq 0$ & $%
+ $ & $-$ \\ \hline
$\searrow \nearrow $ & $-$ & $-$ & $\nearrow \searrow $ & $0$ & $<0$ & $%
\left( +,-\right) $ & $\left( -,+\right) $ \\ \hline
\end{tabular}
\\ 
\multicolumn{1}{c}{} \\ 
\multicolumn{1}{c}{\text{Table 3: when }f(a^{+})=g(a^{+})=0}%
\end{array}%
\end{equation*}%
Thus we complete the proof.
\end{proof}

Similarly, according to Table 4, we can state and prove Theorem \ref%
{T-Pinelis4-Yb}, whose details are omitted.%
\begin{equation*}
\begin{array}{l}
\begin{tabular}{|c|c|c|c|c|c|c|c|}
\hline
$f^{\prime }/g^{\prime }$ & $g^{\prime }$ & $g$ & $H_{f,g}$ & $H_{f,g}\left(
a^{+}\right) $ & $H_{f,g}\left( b^{-}\right) $ & $\func{sgn}\left(
H_{f,g}\right) $ & $\func{sgn}\left( f/g\right) ^{\prime }$ \\ \hline
$\nearrow \searrow $ & $+$ & $-$ & $\searrow \nearrow $ & $\leq 0$ & $0$ & $%
- $ & $-$ \\ \hline
$\nearrow \searrow $ & $+$ & $-$ & $\searrow \nearrow $ & $>0$ & $0$ & $%
\left( +,-\right) $ & $\left( +,-\right) $ \\ \hline
$\searrow \nearrow $ & $+$ & $-$ & $\nearrow \searrow $ & $\geq 0$ & $0$ & $%
+ $ & $+$ \\ \hline
$\searrow \nearrow $ & $+$ & $-$ & $\nearrow \searrow $ & $<0$ & $0$ & $%
\left( -,+\right) $ & $\left( -,+\right) $ \\ \hline
$\nearrow \searrow $ & $-$ & $+$ & $\nearrow \searrow $ & $\geq 0$ & $0$ & $%
+ $ & $-$ \\ \hline
$\nearrow \searrow $ & $-$ & $+$ & $\nearrow \searrow $ & $<0$ & $0$ & $%
\left( -,+\right) $ & $\left( +,-\right) $ \\ \hline
$\searrow \nearrow $ & $-$ & $+$ & $\searrow \nearrow $ & $\leq 0$ & $0$ & $%
- $ & $+$ \\ \hline
$\searrow \nearrow $ & $-$ & $+$ & $\searrow \nearrow $ & $>0$ & $0$ & $%
\left( +,-\right) $ & $\left( -,+\right) $ \\ \hline
\end{tabular}
\\ 
\multicolumn{1}{c}{} \\ 
\multicolumn{1}{c}{\text{Table 4: when }f(b^{-})=g(b^{-})=0}%
\end{array}%
\end{equation*}

\begin{theorem}[LPMR]
\label{T-Pinelis4-Yb}Let $-\infty \leq a<b\leq \infty $ and let $H_{f,g}$ be
defined by (\ref{H_f,g}). Suppose that (i) $f$ and $g$ are differentiable
functions on $(a,b)$; (ii) $g^{\prime }$ never vanishes on $(a,b)$; (ii) $%
f(b^{-})=g(b^{-})=0$; (iv) there is a $c\in \left( a,b\right) $ such that $%
f^{\prime }/g^{\prime }$ is increasing (decreasing) on $\left( a,c\right) $
and decreasing (increasing) on $\left( c,b\right) $. Then

(i) when $g^{\prime }>0$ and $H_{f,g}\left( a^{+}\right) \leq \left( \geq
\right) 0$, or $g^{\prime }<0$ and $H_{f,g}\left( a^{+}\right) \geq \left(
\leq \right) 0$, $f/g$ is decreasing (increasing) on $\left( a,b\right) $;

(ii) when $g^{\prime }>0$ and $H_{f,g}\left( a^{+}\right) >\left( <\right) 0$%
, or $g^{\prime }<0$ and $H_{f,g}\left( a^{+}\right) <\left( >\right) 0$,
there is a unique number $x_{b}\in \left( a,b\right) $ such that $f/g$ is
increasing (decreasing) on $\left( a,x_{b}\right) $ and decreasing
(increasing) on $\left( x_{b},b\right) $.
\end{theorem}

The following statement is crucial to prove certain best analytic
inequalities, which is inspired by part (iv) of proof of Theorem 6 in \cite%
{Yang-arXiv-1304.5369} or Theorem 25 in \cite{Yang-AAA-2014-601839} and has
been proven in \cite[Lemma 3]{Yang-arXiv-1408.2250}. Here we give a simple
proof by utilizing Theorem \ref{T-Pinelis4-Ya}.

\begin{corollary}
\label{C-Y}Let $-\infty \leq a<b\leq \infty $. Suppose that (i) $f$ and $g$
are differentiable functions on $(a,b)$; (ii) $g^{\prime }\neq 0$ on $(a,b)$%
; (iii) there is a $c\in \left( a,b\right) $ such that $f^{\prime
}/g^{\prime }$ is increasing (decreasing) on $\left( a,c\right) $ and
decreasing (increasing) on $\left( c,b\right) $. Then

(i) when $f(a^{+})=g(a^{+})=0$ and $\lim_{x\rightarrow a^{+}}\left( f\left(
x\right) /g\left( x\right) \right) =\lambda \neq \pm \infty $, the
inequality $f\left( x\right) /g\left( x\right) >\left( <\right) \lambda $
holds for $x\in \left( a,b\right) $ if and only if $\lim_{x\rightarrow
b^{-}}\left( f\left( x\right) /g\left( x\right) \right) \geq \left( \leq
\right) \lambda $;

(ii) when $f(b^{-})=g(b^{-})=0$ and $\lim_{x\rightarrow b^{-}}\left( f\left(
x\right) /g\left( x\right) \right) =\lambda \neq \pm \infty $, the
inequality $f\left( x\right) /g\left( x\right) >\left( <\right) \lambda $
holds for $x\in \left( a,b\right) $ if and only if $\lim_{x\rightarrow
a^{+}}\left( f\left( x\right) /g\left( x\right) \right) \geq \left( \leq
\right) \lambda $.
\end{corollary}

\begin{proof}
We only prove part (i) of this assertion in the case when $f^{\prime
}/g^{\prime }$ is increasing on $\left( a,c\right) $ and decreasing on $%
\left( c,b\right) $, part (i) in another case and part (ii) can be proved in
similar way.

From Theorem \ref{T-Pinelis4-Ya} (see also Table 3), we distinguish two
cases:

Case 1: $g^{\prime }>0$ and $H_{f,g}\left( b^{-}\right) \geq 0$, or $%
g^{\prime }<0$ and $H_{f,g}\left( b^{-}\right) \leq 0$. In this case, by
Theorem \ref{T-Pinelis4-Ya} we see that $f/g$ is increasing on $\left(
a,b\right) $, so the assertion is clearly true.

Case 2: $g^{\prime }>0$ and $H_{f,g}\left( b^{-}\right) <0$, or $g^{\prime
}<0$ and $H_{f,g}\left( b^{-}\right) >0$. In this case, by Theorem \ref%
{T-Pinelis4-Ya} we see that there is a unique number $x_{a}\in \left(
a,b\right) $ such that $f/g$ is increasing (decreasing) on $\left(
a,x_{a}\right) $ and decreasing (increasing) on $\left( x_{a},b\right) $. So
it is obtained that%
\begin{eqnarray*}
\frac{f\left( x\right) }{g\left( x\right) } &>&\lim_{x\rightarrow a^{+}}%
\frac{f\left( x\right) }{g\left( x\right) }=\lambda \text{ for }x\in \left(
a,x_{a}\right) , \\
\frac{f\left( x\right) }{g\left( x\right) } &>&\lim_{x\rightarrow b^{-}}%
\frac{f\left( x\right) }{g\left( x\right) }\text{ for }x\in \left(
x_{a},b\right) .
\end{eqnarray*}%
From this it follows that $f\left( x\right) /g\left( x\right) >\lambda $
holds for $x\in \left( a,b\right) $ if and only if $\lim_{x\rightarrow
b^{-}}\left( f\left( x\right) /g\left( x\right) \right) \geq \lambda $.

This completes the proof.
\end{proof}

\section{Applications of LPMR}

\subsection{A reverse one of Lin's inequality in hyperbolic form.}

For $a,b>0$ with $a\neq b$, the logarithmic mean, identric (exponential)
mean and power mean are defined by%
\begin{eqnarray*}
L &\equiv &L\left( a,b\right) =\frac{a-b}{\ln a-\ln b}\text{, \ \ \ }I\equiv
I\left( a,b\right) =\frac{1}{e}\left( \frac{b^{b}}{a^{a}}\right) ^{1/\left(
b-a\right) }, \\
A_{p} &\equiv &A_{p}\left( a,b\right) =\left( \frac{a^{p}+b^{p}}{2}\right)
^{1/p}\text{ if }p\neq 0\text{ and }A_{0}=\sqrt{ab}=G.
\end{eqnarray*}

Lin's inequality \cite{Lin-AMM-81-1972} states that for positive numbers $%
a,b>0$ with $a\neq b$%
\begin{equation*}
L\left( a,b\right) <A_{1/3}
\end{equation*}%
(see also \cite{Yang-MIA-10(3)-2007}, \cite{Yang-JIA-2008-149286}, \cite%
{Yang-BKMS-49(1)-2012}). Let $x=\ln \sqrt{a/b}$. Then Lin's inequality can
be changed into%
\begin{equation}
\frac{\sinh x}{x}<\cosh ^{3}\frac{x}{3}  \label{L-h}
\end{equation}%
for $x>0$ (see also \cite{Neuman-AMC-218-2012}, \cite{Yang-JIA-2013-116}). A
trigonometric version of (\ref{L-h}) can be found in \cite%
{Klen-JIA-2010-362548}, \cite{Yang-JMI-7(4)-2013}, \cite{Yang-GJM-2(1)-2014}.

Now we establish the sharp double inequality related to (\ref{L-h}) as
follows.

\begin{proposition}
\label{P-Lh-Y}For $x>0$, the double inequality%
\begin{equation}
\left( \cosh \frac{x}{3}\right) ^{3p_{0}}<\frac{\sinh x}{x}<\left( \cosh 
\frac{x}{3}\right) ^{3}  \label{LY-h}
\end{equation}%
holds with the best exponents $3$ and $3p_{0}$, where%
\begin{equation*}
p_{0}=\frac{\ln \left( \frac{\sinh x_{0}}{x_{0}}\right) }{3\ln \left( \cosh 
\frac{x_{0}}{3}\right) }\approx 0.88071,
\end{equation*}%
here $x_{0}\approx 7.725532796173$ is the unique root of the equation%
\begin{equation}
3\left( x\coth x-1\right) \coth \tfrac{x}{3}\ln (\cosh \tfrac{x}{3})-x\ln 
\frac{\sinh x}{x}=0  \label{E1}
\end{equation}%
on $\left( 0,\infty \right) $.
\end{proposition}

To prove this proposition, we need the following lemma. Its proof is similar
to Lemma 5 in \cite{Yang-AAA-2014-702718}, here we omit details of proof.

\begin{lemma}[{\protect\cite[Lemma 5]{Yang-AAA-2014-702718}}]
\label{Lemma zp}Let $P\left( x\right) $ be a power series which is
convergent on $\left( 0,\infty \right) $ defined by%
\begin{equation*}
P\left( x\right) =\sum_{i=m+1}^{\infty }a_{i}x^{i}-\sum_{i=0}^{m}a_{i}x^{i},
\end{equation*}%
where $a_{i}\geq 0$ for $i\geq m+1$ with $\max_{i\geq m+1}\left(
a_{i}\right) >0$ and $a_{m}>0$, $a_{i}\geq 0$ for $0\leq i\leq m-1$. Then
there is a unique number $x_{0}\in \left( 0,\infty \right) $ to satisfy $%
P\left( x\right) =0$ such that $P\left( x\right) <0$ for $x\in \left(
0,x_{0}\right) $ and $P\left( x\right) >0$ for $x\in \left( x_{0},\infty
\right) $.
\end{lemma}

\begin{proof}[Proof of Proposition \protect\ref{P-Lh-Y}]
For $x\in \left( 0,\infty \right) $, we define 
\begin{equation*}
f\left( x\right) =\ln \frac{\sinh x}{x}\text{ \ and \ }g\left( x\right)
=3\ln \cosh \frac{x}{3}.
\end{equation*}%
It is easy to check that $g^{\prime }\left( x\right) =\tanh \left(
x/3\right) >0$ on $\left( 0,\infty \right) $ and clearly, $f\left(
0^{+}\right) =g\left( 0^{+}\right) =0$. Now we show that there is a $%
x_{1}\in \left( 0,\infty \right) $ such that $f^{\prime }/g^{\prime }$ is
decreasing on $\left( 0,x_{1}\right) $ and increasing on $\left(
x_{1},\infty \right) $.

Differentiation gives%
\begin{eqnarray*}
\frac{f^{\prime }\left( x\right) }{g^{\prime }\left( x\right) } &=&\frac{%
\frac{\cosh x}{\sinh x}-\frac{1}{x}}{\frac{\sinh \frac{x}{3}}{\cosh \frac{x}{%
3}}}=\coth \frac{x}{3}\coth x-\frac{1}{x}\coth \frac{x}{3}, \\
\left( \frac{f^{\prime }\left( x\right) }{g^{\prime }\left( x\right) }%
\right) ^{\prime } &=&-\frac{\cosh \frac{x}{3}}{\sinh \frac{x}{3}\sinh ^{2}x}%
-\frac{1}{3}\frac{\cosh x}{\sinh ^{2}\frac{x}{3}\sinh x}+\frac{1}{3}\frac{1}{%
x\sinh ^{2}\frac{x}{3}}+\frac{1}{x^{2}}\frac{\cosh \frac{x}{3}}{\sinh \frac{x%
}{3}} \\
&=&\frac{\left( -3\cosh \frac{x}{3}\sinh \frac{x}{3}-\cosh x\sinh x\right)
x^{2}+x\sinh ^{2}x+3\cosh \frac{x}{3}\sinh \frac{x}{3}\sinh ^{2}x}{%
3x^{2}\sinh ^{2}\frac{x}{3}\sinh ^{2}x} \\
&:&=\frac{h\left( x\right) }{3x^{2}\sinh ^{2}\frac{x}{3}\sinh ^{2}x}.
\end{eqnarray*}%
Using "product into sum" formulas for hyperbolic functions and expanding in
power series, we have%
\begin{eqnarray*}
\frac{8}{3}h\left( \frac{3x}{2}\right) &=&\sinh 4x-3x^{2}\sinh 3x+2x\cosh
3x-9x^{2}\sinh x-\sinh 2x-2\sinh x-2x \\
&=&\sum_{n=1}^{\infty }\frac{4^{2n-1}}{\left( 2n-1\right) !}%
x^{2n-1}-3\sum_{n=2}^{\infty }\frac{3^{2n-3}}{\left( 2n-3\right) !}%
x^{2n-1}+2\sum_{n=1}^{\infty }\frac{3^{2n-2}}{\left( 2n-2\right) !}x^{2n-1}
\\
&&-\sum_{n=1}^{\infty }\frac{2^{2n-1}}{\left( 2n-1\right) !}%
x^{2n-1}-9\sum_{n=2}^{\infty }\frac{1}{\left( 2n-3\right) !}%
x^{2n-1}-2\sum_{n=1}^{\infty }\frac{1}{\left( 2n-1\right) !}x^{2n-1}-2x \\
&:&=\sum_{n=2}^{\infty }\frac{u_{n}}{\left( 2n-1\right) !}x^{2n-1},
\end{eqnarray*}

where%
\begin{equation*}
u_{n}=4^{2n-1}-2^{2n-1}-2\left( 2n-1\right) \left( n-2\right)
3^{2n-2}-2\left( 3n-2\right) \left( 6n-5\right) .
\end{equation*}

It is easy to check that $u_{2}=u_{3}=0$, $u_{n}<0$ for $n=4,5,6,7$. While
the recursive relation%
\begin{equation*}
u_{n+1}-16u_{n}=2\left( 14n^{2}-71n+41\right) 3^{2n-2}+6\times
2^{2n}+6\left( 90n^{2}-147n+53\right) >0
\end{equation*}%
for $n\geq 8$ together with $u_{8}=212\,772\,744>0$ leads to $u_{n}>0$ for $%
n\geq 8$.

By Lemma \ref{Lemma zp} we see that there is a unique $x_{2}\in \left(
0,\infty \right) $ such that $h\left( 3x/2\right) <0$ for $x\in \left(
0,x_{2}\right) $ and $h\left( 3x/2\right) >0$ for $x\in \left( x_{2},\infty
\right) $. This implies that $f^{\prime }/g^{\prime }$ is decreasing on $%
\left( 0,x_{1}\right) $ and increasing on $\left( x_{1},\infty \right) $,
where $x_{1}=3x_{2}/2$.

On the other hand, we see that%
\begin{eqnarray*}
H_{f,g}\left( x\right) &=&\frac{f^{\prime }\left( x\right) }{g^{\prime
}\left( x\right) }g\left( x\right) -f\left( x\right) \\
&=&\left( \coth \frac{x}{3}\coth x-\frac{1}{x}\coth \frac{x}{3}\right) 3\ln
\cosh x-\ln \frac{\sinh x}{x}.
\end{eqnarray*}%
Direct computations lead us to $H_{f,g}\left( 0^{+}\right) =0$ and%
\begin{eqnarray*}
\lim_{x\rightarrow \infty }H_{f,g}\left( x\right) &=&\lim_{x\rightarrow
\infty }\left( 3\left( \coth \frac{x}{3}\coth x-\frac{1}{x}\coth \frac{x}{3}%
\right) \ln \cosh \frac{x}{3}-\ln \frac{\sinh x}{x}\right) \\
&=&\infty .
\end{eqnarray*}

Application of part (ii) of Theorem \ref{T-Pinelis4-Ya} reveals that there
is a unique number $x_{0}\in \left( 0,\infty \right) $ such that $f/g$ is
decreasing on $\left( 0,x_{0}\right) $ and increasing on $\left(
x_{0},\infty \right) $. Therefore, we conclude that%
\begin{eqnarray*}
\frac{f\left( x_{0}\right) }{g\left( x_{0}\right) } &\leq &\frac{f\left(
x\right) }{g\left( x\right) }<\lim_{x\rightarrow 0^{+}}\frac{f\left(
x\right) }{g\left( x\right) }=1\text{ for }x\in (0,x_{0}], \\
\frac{f\left( x_{0}\right) }{g\left( x_{0}\right) } &<&\frac{f\left(
x\right) }{g\left( x\right) }<\lim_{x\rightarrow \infty }\frac{f\left(
x\right) }{g\left( x\right) }=1\text{ for }x\in \left( x_{0},\infty \right) ,
\end{eqnarray*}%
that is,%
\begin{equation*}
p_{0}=\frac{\ln \frac{\sinh x_{0}}{x_{0}}}{3\ln \cosh \frac{x_{0}}{3}}\leq 
\frac{\ln \frac{\sinh x}{x}}{3\ln \cosh \frac{x}{3}}\leq 1,
\end{equation*}%
which proves the desired companion inequality (\ref{LY-h}).

Solving the equation%
\begin{equation*}
H_{f,g}\left( x\right) =\left( \coth \frac{x}{3}\coth x-\frac{1}{x}\coth 
\frac{x}{3}\right) 3\ln \cosh x-\ln \frac{\sinh x}{x}=0,
\end{equation*}%
which is equivalent to (\ref{E1}), we find that $x_{0}\approx 7.725532796173$%
, and then%
\begin{equation*}
p_{0}=\frac{\ln \frac{\sinh x_{0}}{x_{0}}}{3\ln \cosh \frac{x_{0}}{3}}%
\approx 0.88071.
\end{equation*}

This completes the proof.
\end{proof}

\begin{remark}
For $a,b>0$ with $a\neq b$, let $x=\ln \sqrt{a/b}$ in (\ref{LY-h}). Then (%
\ref{LY-h}) can be written as%
\begin{equation*}
A_{1/3}^{p_{0}}G^{1-p_{0}}<L<A_{1/3},
\end{equation*}%
where $p_{0}$ is the best possible.
\end{remark}

\begin{remark}
Similarly, for $a,b>0$ with $a\neq b$, Stolarsky's inequality \cite%
{Stolarsky-AMM-87-1980}%
\begin{equation}
I\left( a,b\right) =\frac{1}{e}\left( \frac{b^{b}}{a^{a}}\right) ^{1/\left(
b-a\right) }>\left( \frac{a^{2/3}+b^{2/3}}{2}\right) ^{3/2}=A_{2/3}
\label{I-Ap}
\end{equation}%
can be changed into%
\begin{equation}
e^{\frac{x\cosh x}{\sinh x}-1}>\left( \cosh \frac{2x}{3}\right) ^{3/2}
\label{S-h}
\end{equation}%
for $x>0$. In the same method, we can establish a sharp double inequality
related to (\ref{S-h}), further proof details of which is left to readers.
\end{remark}

\begin{proposition}
For $x>0$, the double inequality%
\begin{equation*}
\left( \cosh \frac{2x}{3}\right) ^{3/2}<e^{\frac{x\cosh x}{\sinh x}%
-1}<\left( \cosh \frac{2x}{3}\right) ^{3p_{1}/2}
\end{equation*}%
holds with the best constants $1$ and 
\begin{equation*}
p_{1}=\frac{\frac{x_{0}\cosh x_{0}}{\sinh x_{0}}-1}{\frac{3}{2}\ln \left(
\cosh \frac{2x_{0}}{3}\right) }\approx 1.0140,
\end{equation*}%
where $x_{0}\approx 2.672067684303$ is the unique root of the equation%
\begin{equation*}
\tfrac{3}{2}\left( \cosh x\sinh x-x\right) \coth \frac{2x}{3}\ln \left(
\cosh \tfrac{2x}{3}\right) -\left( x\cosh x-\sinh x\right) \sinh x=0
\end{equation*}%
on $\left( 0,\infty \right) $.
\end{proposition}

\subsection{New sharp inequalities for identric (exponential) mean.}

There has many inequalities between identric (exponential) mean and power
mean, here we only quote those sharp ones. In \cite%
{Pittinger-UBPEFSMF-680-1980}, Pittinger proved the inequality%
\begin{equation*}
I\left( a,b\right) <A_{\ln 2}\left( a,b\right)
\end{equation*}%
holds, where $\ln 2$ is the best. Alzer \cite{Alzer-EM-43-1988} and Neuman
and S\'{a}ndor \cite{Neuman-IJMMS-2003-16} showed the double inequality%
\begin{equation*}
\frac{2}{e}A_{1}\left( a,b\right) <I\left( a,b\right) <\frac{4}{e}%
A_{1/2}\left( a,b\right)
\end{equation*}%
holds with best coefficients $2/e$ and $4/e$ (see also \cite%
{Yang-JIPAM-6(4)-2005}). The following sharp inequality%
\begin{equation*}
I\left( a,b\right) <\frac{2\sqrt{2}}{e}A_{2/3}\left( a,b\right)
\end{equation*}%
is due to Yang \cite{Yang-MIA-10(3)-2007} (see also \cite%
{Neuman-AADM-3(1)-2009}).

Now we present a more general results involving identric (exponential) mean
and power mean. Without loss of generality, we assume that $b>a>0$, and set $%
x=a/b$. Then $x\in \left( 0,1\right) $. We have

\begin{proposition}
\label{P-IP-Y}Let $p\in \mathbb{R}$ and $x\in \left( 0,1\right) $.

(i) For $p\in (-\infty ,2/3]$, the inequalities%
\begin{equation}
A_{p}\left( x,1\right) <I\left( x,1\right) <A_{p}\left( x,1\right) ^{p/\ln 2}
\label{I-new1}
\end{equation}%
hold for $x\in \left( 0,1\right) $ with the best exponents $1$ and $p/\ln 2$%
. For $p\in \lbrack 1,\infty )$, (\ref{I-new1}) are reversed.

(ii) $p\in (2/3,1)$, the double inequality%
\begin{equation}
A_{p}\left( x,1\right) ^{\delta _{0}\left( p\right) }\leq I\left( x,1\right)
<A_{p}\left( x,1\right) ^{\delta _{1}\left( p\right) }  \label{I-new2}
\end{equation}%
holds for $x\in \left( 0,1\right) $, where%
\begin{equation*}
\delta _{0}\left( p\right) =\frac{\frac{x_{0}\ln x_{0}}{x_{0}-1}-1}{\frac{1}{%
p}\ln \frac{x_{0}^{p}+1}{2}}\text{ \ and \ }\delta _{1}\left( p\right)
=\min_{p\in \left( 2/3,1\right) }\left( \frac{p}{\ln 2},1\right)
\end{equation*}%
are the best exponents, here $x_{0}$ is the unique root of the equation%
\begin{equation}
H_{f,g}\left( x\right) =\frac{\left( x+x^{1-p}\right) \left( x-1-\ln
x\right) }{\left( x-1\right) ^{2}}\frac{1}{p}\ln \frac{x^{p}+1}{2}-\left( 
\frac{x\ln x}{x-1}-1\right) =0  \label{H=0-I}
\end{equation}%
on $\left( 0,1\right) $. In particular, let $p=\ln 2$. Then we have%
\begin{equation}
A_{\ln 2}\left( x,1\right) ^{\delta _{0}\left( \ln 2\right) }<I\left(
x,1\right) <A_{\ln 2}\left( x,1\right)  \label{I-new3}
\end{equation}%
holds for $x\in \left( 0,1\right) $, where $\delta _{0}\left( \ln 2\right)
\approx 1.0154$.
\end{proposition}

In order to prove the Proposition \ref{P-IP-Y}, the following statement is
necessary.

\begin{lemma}
\label{L-sgnw}Let the function $w$ be defined on $\left( 0,1\right) $ by%
\begin{equation}
w\left( x\right) =x^{p}+p\left( p-1\right) x-\left( p^{2}+2p-2\right)
-p\left( p+1\right) x^{1-p}+\left( p-1\right) ^{2}x^{-p}.  \label{w}
\end{equation}%
Then (i) $w\left( x\right) >0$ if $p\in (0,2/3]$;

(ii) if $p\in (2/3,1)$, then there is a unique number $x_{1}\in \left(
0,1\right) $ such that $w\left( x\right) >0$ for $x\in \left( 0,x_{1}\right) 
$ and $w\left( x\right) <0$ for $x\in \left( x_{1},1\right) $;

(iii) $w\left( x\right) <0$ if $p=1$.
\end{lemma}

\begin{proof}
For $p\in \left( 0,1\right) $, differentiations give%
\begin{eqnarray}
w^{\prime }\left( x\right) &=&\frac{p}{x^{p+1}}\left( x^{2p}+\left(
p-1\right) x^{p+1}-\left( p-1\right) ^{2}+x\left( p+1\right) \left(
p-1\right) \right) ,  \label{dw} \\
w^{\prime }\left( 0^{+}\right) &=&-\infty \text{, \ \ \ }w^{\prime }\left(
1\right) =p\left( 3p-2\right) ,  \notag
\end{eqnarray}%
\begin{equation}
\frac{x^{p+2}}{p\left( p-1\right) }w^{\prime \prime }\left( x\right)
=x^{2p}+\left( p-1\right) \left( p+1\right) -x\left( p^{2}+p\right)
:=w_{1}(x).  \label{ddw}
\end{equation}

By arithmetic geometric mean inequality, we have%
\begin{equation*}
x^{p}<1-p+px,
\end{equation*}%
and then,%
\begin{eqnarray*}
w_{1}(x) &=&x^{2p}+\left( p-1\right) \left( p+1\right) -x\left(
p^{2}+p\right) \\
&<&\left( 1-p+px\right) ^{2}+\left( p-1\right) \left( p+1\right) -x\left(
p^{2}+p\right) \\
&=&-p\left( px+\left( 1-p\right) \right) \left( 2-x\right) <0\text{ for }%
x\in \left( 0,1\right) .
\end{eqnarray*}%
It is derived from (\ref{ddw}) that $w^{\prime \prime }\left( x\right) >0$
on $\left( 0,1\right) $ for $p\in \left( 0,1\right) $.

(i) Now we show that $w\left( x\right) >0$ if $p\in (0,2/3]$. The assertion
that $w^{\prime \prime }\left( x\right) >0$ on $\left( 0,1\right) $ yields
that $w^{\prime }\left( x\right) <w^{\prime }\left( 1\right) =p\left(
3p-2\right) <0$, and hence, $w\left( x\right) >w\left( 1\right) =-2\left(
3p-2\right) \geq 0$.

(ii) If $p\in (2/3,1)$, then by using the convexity of $w$ on $\left(
0,1\right) $ and noting that the facts $w\left( 0^{+}\right) =\infty $ and $%
w\left( 1\right) =-2\left( 3p-2\right) <0$, we see that there is a unique
number $x_{1}\in \left( 0,1\right) $ such that $w\left( x\right) >0$ for $%
x\in \left( 0,x_{1}\right) $ and $w\left( x\right) <0$ for $x\in \left(
x_{1},1\right) $.

(iii) If $p=1$, then $w\left( x\right) =x-3<0$ for $x\in \left( 0,1\right) $.

This completes the proof of this lemma.
\end{proof}

We now are in a position to prove Proposition \ref{P-IP-Y}.

\begin{proof}[Proof of Proposition \protect\ref{P-IP-Y}]
For $x\in \left( 0,1\right) $, we define%
\begin{equation*}
f\left( x\right) =\ln I\left( x,1\right) =\frac{x\ln x}{x-1}-1\text{, \ \ \ }%
g\left( x\right) =\ln A_{p}\left( x,1\right) =\frac{1}{p}\ln \frac{x^{p}+1}{2%
}.
\end{equation*}%
Differentiation gives%
\begin{eqnarray}
\frac{f^{\prime }\left( x\right) }{g^{\prime }\left( x\right) } &=&\frac{%
x+x^{1-p}}{\left( x-1\right) ^{2}}\left( x-1-\ln x\right)  \notag \\
&=&\frac{x-1-\ln x}{\left( x-1\right) ^{2}/\left( x+x^{1-p}\right) }:=\frac{%
f_{1}\left( x\right) }{g_{1}\left( x\right) },  \notag \\
\frac{f_{1}^{\prime }\left( x\right) }{g_{1}^{\prime }\left( x\right) } &=&%
\frac{x^{p-1}\left( x+x^{1-p}\right) ^{2}}{x^{p+1}+x^{p}+\left( p+1\right)
x+1-p},  \notag \\
\left( \frac{f_{1}^{\prime }\left( x\right) }{g_{1}^{\prime }\left( x\right) 
}\right) ^{\prime } &=&\frac{\left( x^{p}+1\right) \times w\left( x\right) }{%
\left( x^{p+1}+x^{p}+\left( p+1\right) x+1-p\right) ^{2}},
\label{d(df1/dg1)}
\end{eqnarray}%
where $w\left( x\right) $ is defined by (\ref{w}).

(i) In the case when $p\in (-\infty ,2/3]$ or $p\in \lbrack 1,\infty )$. We
first prove inequalities ((\ref{I-new1})) holds for $p\in (0,2/3]$. By Lemma %
\ref{L-sgnw} we see that $w\left( x\right) >0$ on $(0,1)$, which shows that $%
f_{1}^{\prime }/g_{1}^{\prime }$ is increasing on $\left( 0,1\right) $.
Since $f_{1}\left( 1^{-}\right) =g_{1}\left( 1^{-}\right) =0$, by Theorem %
\ref{T-Anderson} or \ref{T-Pinelis1}, we see that $f_{1}/g_{1}=f^{\prime
}/g^{\prime }$ is also increasing on $\left( 0,1\right) $, which in turn
implies that $f/g$ is so. Consequently, we conclude that%
\begin{equation*}
\frac{p}{\ln 2}=\lim_{x\rightarrow 0^{+}}\frac{f\left( x\right) }{g\left(
x\right) }<\frac{f\left( x\right) }{g\left( x\right) }=\frac{\frac{x\ln x}{%
x-1}-1}{\frac{1}{p}\ln \frac{x^{p}+1}{2}}<\lim_{x\rightarrow 1^{-}}\frac{%
f\left( x\right) }{g\left( x\right) }=1,
\end{equation*}%
which proves (\ref{I-new1}).

Since $p\mapsto A_{p}\left( x,1\right) $ is increasing on $\mathbb{R}$, and $%
p\mapsto A_{p}\left( x,1\right) ^{p/\ln 2}=\left( \left( x^{p}+1\right)
/2\right) ^{1/\ln 2}$ is clearly decreasing on $\mathbb{R}$ for $x\in \left(
0,1\right) $, the double inequality (\ref{I-new1}) is still valid for $p\in
(-\infty ,0]$.

In the same argument, in order to show that the reverse of (\ref{I-new1})
holds in the case when $p\geq 1$, it suffices to consider the case $p=1$. In
fact, by Lemma \ref{L-sgnw} we have $w\left( x\right) <0$ for $x\in \left(
0,1\right) $. Similar to part one of this proof, we easily deduce that $f/g$
is decreasing on $\left( 0,1\right) $, which yields 
\begin{equation*}
1=\lim_{x\rightarrow 1^{-}}\frac{f\left( x\right) }{g\left( x\right) }<\frac{%
f\left( x\right) }{g\left( x\right) }=\frac{\frac{x\ln x}{x-1}-1}{\ln \frac{%
x+1}{2}}<\lim_{x\rightarrow 0^{+}}\frac{f\left( x\right) }{g\left( x\right) }%
=\frac{1}{\ln 2},
\end{equation*}%
that is, the reverse of (\ref{I-new1}) holds when $p=1$.

(ii) In the case when $p\in \left( 2/3,1\right) $. By Lemma \ref{L-sgnw},
there is a unique number $x_{1}\in \left( 0,1\right) $ such that $w\left(
x\right) >0$ for $x\in \left( 0,x_{1}\right) $ and $w\left( x\right) <0$ for 
$x\in \left( x_{1},1\right) $. It follows from (\ref{d(df1/dg1)}) that $%
f_{1}^{\prime }/g_{1}^{\prime }$ is increasing on $\left( 0,x_{1}\right) $
and decreasing on $\left( x_{1},1\right) $. Simple verifications give%
\begin{equation*}
g_{1}^{\prime }\left( x\right) =\frac{\left( x-1\right) \left(
x^{p+1}+x^{p}+\left( p+1\right) x+1-p\right) }{x^{p}\left( x+x^{1-p}\right)
^{2}}<0\text{ on }\left( 0,1\right)
\end{equation*}%
and%
\begin{eqnarray*}
H_{f_{1},g_{1}}\left( x\right) &:&=\frac{f_{1}^{\prime }\left( x\right) }{%
g_{1}^{\prime }\left( x\right) }g_{1}\left( x\right) -f_{1}\left( x\right) \\
&=&\frac{x^{p-1}\left( x+x^{1-p}\right) ^{2}}{x^{p+1}+x^{p}+\left(
p+1\right) x+1-p}\frac{\left( x-1\right) ^{2}}{x+x^{1-p}}-\left( x-1-\ln
x\right) \\
&=&\frac{\left( x^{p}+1\right) \left( x-1\right) ^{2}}{x^{p+1}+x^{p}+\left(
p+1\right) x+1-p}-\left( x-1-\ln x\right) ,
\end{eqnarray*}%
and then, $H_{f_{1},g_{1}}\left( 1^{-}\right) =0$ and $\lim_{x\rightarrow
0^{+}}H_{f_{1},g_{1}}\left( x\right) =-\infty $. Utilizing part (ii) of
Theorem \ref{T-Pinelis4-Yb} we deduce that there is a unique number $%
x_{1}^{\ast }\in \left( 0,1\right) $ such that $f_{1}/g_{1}\left( =f^{\prime
}/g^{\prime }\right) $ is increasing on $\left( 0,x_{1}^{\ast }\right) $ and
decreasing on $\left( x_{1}^{\ast },1\right) $.

It is easy to check that $g^{\prime }\left( x\right) =x^{p-1}/\left(
x^{p}+1\right) >0$ on $\left( 0,1\right) $ and clearly, $f\left(
1^{-}\right) =g\left( 1^{-}\right) =0$. Also, we have%
\begin{eqnarray*}
H_{f,g}\left( x\right) &:&=\frac{\left( x+x^{1-p}\right) \left( x-1-\ln
x\right) }{\left( x-1\right) ^{2}}\frac{1}{p}\ln \frac{x^{p}+1}{2}-\left( 
\frac{x\ln x}{x-1}-1\right) \\
&=&\frac{\left( x^{2/3}+x^{2/3-p}\right) \left( \left( x-1\right)
x^{1/3}-x^{1/3}\ln x\right) }{\left( x-1\right) ^{2}}\frac{1}{p}\ln \frac{%
x^{p}+1}{2}-\left( \frac{x\ln x}{x-1}-1\right) ,
\end{eqnarray*}%
and then $H_{f,g}\left( 1^{-}\right) =0$ and $\lim_{x\rightarrow
0^{+}}H_{f,g}\left( x\right) =1>0$. Utilizing part (ii) of Theorem \ref%
{T-Pinelis4-Yb} again, we infer that there is a unique number $x_{0}\in
\left( 0,1\right) $ such that $f/g$ is increasing on $\left( 0,x_{0}\right) $
and decreasing on $\left( x_{0},1\right) $. Then, we obtain that 
\begin{eqnarray*}
\frac{p}{\ln 2} &=&\lim_{x\rightarrow 0^{+}}\frac{f\left( x\right) }{g\left(
x\right) }<\frac{f\left( x\right) }{g\left( x\right) }<\frac{f\left(
x_{0}\right) }{g\left( x_{0}\right) }\text{ for }x\in \left( 0,x_{0}\right) ,
\\
1 &=&\lim_{x\rightarrow 1^{-}}\frac{f\left( x\right) }{g\left( x\right) }<%
\frac{f\left( x\right) }{g\left( x\right) }<\frac{f\left( x_{0}\right) }{%
g\left( x_{0}\right) }\text{ for }x\in \left( x_{0},1\right) ,
\end{eqnarray*}%
that is,%
\begin{equation*}
\delta _{1}\left( p\right) =\min_{p\in \left( 2/3,1\right) }\left( \frac{p}{%
\ln 2},1\right) <\frac{\frac{x\ln x}{x-1}-1}{\frac{1}{p}\ln \frac{x^{p}+1}{2}%
}\leq \frac{\frac{x_{0}\ln x_{0}}{x_{0}-1}-1}{\frac{1}{p}\ln \frac{%
x_{0}^{p}+1}{2}}=\delta _{0}\left( p\right) \text{ for all }x\in \left(
0,1\right) ,
\end{equation*}%
which proves (\ref{I-new2}).

In particular, letting $p=\ln 2$ yields (\ref{I-new3}). Solving the equation
(\ref{H=0-I}) by mathematical software, we find that $x_{0}\approx 0.0463812$
and $\delta _{0}\left( \ln 2\right) \approx 1.0154$.

This proposition is proved.
\end{proof}

\begin{corollary}
\label{C-IP-Y}For $p\in (0,2/3]$, the inequalities%
\begin{equation}
A_{p}\left( x,1\right) <I\left( x,1\right) <A_{p}\left( x,1\right) ^{p/\ln
2}<e^{-1}2^{1/p}A_{p}\left( x,1\right)  \label{I-new4}
\end{equation}%
hold for $x\in \left( 0,1\right) $ with the best coefficients $1$ and $%
e^{-1}2^{1/p}$. Also, the first and third members in (\ref{I-new4}) are
respectively increasing and decreasing in $p$ on $\mathbb{R}$, while the
fourth one is decreasing in $p$ on $\left( 0,\infty \right) $. For $p\in
\lbrack 1,\infty )$, (\ref{I-new4}) are reversed.
\end{corollary}

\begin{proof}
To prove (\ref{I-new4}) holds, it is enough to prove that $\ln \left(
e^{-1}2^{1/p}A_{p}\left( x,1\right) \right) -\ln A_{p}\left( x,1\right)
^{p/\ln 2}>\left( <\right) 0$ for $p\in (0,2/3]$ ($[1,\infty )$). Indeed,
simplifying yields%
\begin{equation*}
\ln \left( e^{-1}2^{1/p}A_{p}\left( x,1\right) \right) -\ln A_{p}\left(
x,1\right) ^{p/\ln 2}=\frac{\ln 2-p}{p\ln 2}\left( \ln 2+p\ln A_{p}\left(
x,1\right) \right) =\frac{\ln 2-p}{p\ln 2}\ln \left( x^{p}+1\right) ,
\end{equation*}%
which is obviously positive if $p\in (0,2/3]$ and negative if $p\geq 1$ for $%
x\in \left( 0,1\right) $.

The following limit relations%
\begin{equation*}
\lim_{x\rightarrow 1^{-}}\frac{I\left( x,1\right) }{A_{p}\left( x,1\right) }%
=1\text{ \ and \ }\lim_{x\rightarrow 0^{+}}\frac{I\left( x,1\right) }{%
A_{p}\left( x,1\right) }=e^{-1}2^{1/p}
\end{equation*}%
show that the coefficients in (\ref{I-new4}), that is, $1$ and $%
e^{-1}2^{1/p} $, are the best.

As mentioned in proof of Proposition \ref{P-IP-Y}, the first and third
members in (\ref{I-new4}) are respectively increasing and decreasing in $p$
on $\mathbb{R}$. While the decreasing property of the fourth one in $p$ on $%
\left( 0,\infty \right) $ follows from Lemma 6 in \cite{Yang-JMI-7(4)-2013E}.
\end{proof}

\subsection{A supplement to a result proved by Zhu.}

Zhu was early aware of the special role of the L'Hospital monotonicity rule
in establishing new inequalities for trigonometric, for example, as early as
2004, he \cite{Zhu-AML-19-2006} first used the LMR to give a sharpening
Jordan's inequality. After this, he published a series of results for
trigonometric and hyperbolic functions as well as bivariate means (see \cite%
{Zhu-JIA-2007-67430}, \cite[Theorem 1]{Zhu-CMA-58-2009}, \cite%
{Zhu-AAA-2009-485842}). In 2009, he \cite[Theorem 1]{Zhu-CMA-58-2009} (see
also \cite{Zhu-AAA-2009-485842}) established a general result related to
Cusa-type inequalities. For own convenience, we record it as follows.

\noindent \textbf{Theorem Zhu (}\cite{Zhu-CMA-58-2009}\textbf{)}\label{Zhu}%
\emph{Let }$0<x<\pi /2$\emph{. Then}

\emph{(i) if }$p\geq 1$\emph{, the double inequality}%
\begin{equation}
1-\xi +\xi \left( \cos x\right) ^{p}<\left( \frac{\sin x}{x}\right)
^{p}<1-\eta +\eta \left( \cos x\right) ^{p}  \label{Zhu1}
\end{equation}%
\emph{holds if and only if }$\eta \leq 1/3$\emph{\ and }$\xi \geq 1-\left(
2/\pi \right) ^{p}$\emph{;}

\emph{(ii) if }$0\leq p\leq 4/5$\emph{, the double inequality}%
\begin{equation}
1-\eta +\eta \left( \cos x\right) ^{p}<\left( \frac{\sin x}{x}\right)
^{p}<1-\xi +\xi \left( \cos x\right) ^{p}  \label{Zhu2}
\end{equation}%
\emph{holds if and only if }$\eta \geq 1/3$\emph{\ and }$\xi \leq 1-\left(
2/\pi \right) ^{p}$\emph{'}

\emph{(iii) if }$p<0$\emph{, the inequality}%
\begin{equation*}
\left( \frac{\sin x}{x}\right) ^{p}<1-\eta +\eta \left( \cos x\right) ^{p}
\end{equation*}%
\emph{holds if and only if }$\eta \geq 1/3$\emph{.}

Now we give a supplement to Theorem \ref{Zhu} in the case when $p\in \left(
4/5,1\right) $.

\begin{proposition}
Let $x\in (0,\pi /2)$ and $p\in \left( 4/5,1\right) $. Then the double
inequality%
\begin{equation}
\theta _{0}\left( p\right) \left( \cos x\right) ^{p}+1-\theta _{0}\left(
p\right) \leq \left( \frac{\sin x}{x}\right) ^{p}<\theta _{1}\left( p\right)
\left( \cos x\right) ^{p}+1-\theta _{1}\left( p\right)  \label{Z-Y1}
\end{equation}%
holds true, where%
\begin{equation*}
\theta _{0}\left( p\right) =\frac{\left( \frac{\sin x_{0}}{x_{0}}\right)
^{p}-1}{\left( \cos x_{0}\right) ^{p}-1}\text{ \ and \ }\theta _{1}\left(
p\right) =\min_{p\in \left( 4/5,1\right) }\left( 1-\left( \frac{2}{\pi }%
\right) ^{p},\frac{1}{3}\right)
\end{equation*}%
are the best constants, here $x_{0}$ is the unique root of the equation%
\begin{equation}
\left( \frac{\sin x}{x\cos x}\right) ^{p-1}\frac{\sin x-x\cos x}{x^{2}\sin x}%
\left( \left( \cos x\right) ^{p}-1\right) -\left( \left( \frac{\sin x}{x}%
\right) ^{p}-1\right) =0  \label{H=0-t}
\end{equation}%
on $(0,\pi /2)$. Particularly, when $1-\left( 2/\pi \right) ^{p}=1/3$, that
is,%
\begin{equation*}
p=p_{0}=\frac{\ln 3-\ln 2}{\ln \pi -\ln 2}\approx 0.89788,\text{ (see \cite%
{Yang-MIA-17(2)-2014})}
\end{equation*}%
we have%
\begin{equation}
\theta _{0}\left( p_{0}\right) \left( \cos x\right) ^{p_{0}}+1-\theta
_{0}\left( p_{0}\right) <\left( \frac{\sin x}{x}\right) ^{p_{0}}\leq \frac{1%
}{3}\left( \cos x\right) ^{p_{0}}+\frac{2}{3},  \label{Z-Y2}
\end{equation}%
where $\theta _{0}\left( p_{0}\right) \approx 0.33334$.
\end{proposition}

\begin{proof}
For $x\in (0,\pi /2)$, we define $f$ and $g$ as%
\begin{equation*}
f\left( x\right) =\left( \frac{\sin x}{x}\right) ^{p}-1\text{ \ and \ }%
g\left( x\right) =\left( \cos x\right) ^{p}-1.
\end{equation*}%
Then%
\begin{equation*}
\frac{f^{\prime }\left( x\right) }{g^{\prime }\left( x\right) }=\left( \frac{%
\sin x}{x\cos x}\right) ^{p-1}\frac{\sin x-x\cos x}{x^{2}\sin x}
\end{equation*}%
and%
\begin{equation*}
\left( \frac{f^{\prime }\left( x\right) }{g^{\prime }\left( x\right) }%
\right) ^{\prime }=\left( \frac{\sin x}{x\cos x}\right) ^{p-2}\frac{\left(
x-\cos x\sin x\right) \left( \sin x-x\cos x\right) }{\sin x\cos ^{2}x}\left(
p-G\left( x\right) \right) ,
\end{equation*}%
where%
\begin{equation*}
G\left( x\right) =\frac{x\sin x+\cos x\sin ^{2}x-2x^{2}\cos x}{\left( x-\cos
x\sin x\right) \left( \sin x-x\cos x\right) }.
\end{equation*}%
(By the way, in \cite[(21)]{Zhu-CMA-58-2009}, \cite[(4.1)]%
{Zhu-AAA-2009-485842}, the denominator of the second member on the right
hand side should be $\sin x\cos ^{2}x$ instead of $x^{4}\sin x\cos ^{2}x$.
Fortunately, this clerical error did not interfere with the correctness of
Theorem 1.)

It has been shown in \cite[Proof of Theorem 1]{Zhu-CMA-58-2009} that the
function $G$ is increasing on $\left( 0,\pi /2\right) $, and $G\left(
0^{+}\right) =4/5$ and $G\left( \pi /2^{-}\right) =1$. Therefore, the
function $x\mapsto \left( p-G\left( x\right) \right) :=G_{1}\left( x\right) $
is decreasing on $\left( 0,\pi /2\right) $, and $G_{1}\left( 0^{+}\right)
=p-4/5>0$ and $G_{1}\left( \pi /2^{-}\right) =p-1<0$, which reveals that
there is a $x_{1}\in \left( 0,\pi /2\right) $ such that $G_{1}\left(
x\right) >0$ for $x\in \left( 0,x_{1}\right) $ and $G_{1}\left( x\right) <0$
for $x\in \left( x_{1},\pi /2\right) $. Due to $\left( x-\cos x\sin x\right)
>0$ and $\left( \sin x-x\cos x\right) >0$ for $x\in \left( 0,\pi /2\right) $%
, this implies that $f^{\prime }/g^{\prime }$ is increasing on $\left(
0,x_{1}\right) $ and decreasing on $\left( x_{1},\pi /2\right) $.

It is easy to check that $g^{\prime }\left( x\right) =-p\left( \cos x\right)
^{p-1}\sin x<0$, $f\left( 0^{+}\right) =g\left( 0^{+}\right) =0$, and%
\begin{eqnarray*}
H_{f,g}\left( x\right) &=&\frac{f^{\prime }\left( x\right) }{g^{\prime
}\left( x\right) }g\left( x\right) -f\left( x\right) \\
&=&\left( \frac{\sin x}{x\cos x}\right) ^{p-1}\frac{\sin x-x\cos x}{%
x^{2}\sin x}\left( \left( \cos x\right) ^{p}-1\right) -\left( \left( \frac{%
\sin x}{x}\right) ^{p}-1\right) ,
\end{eqnarray*}%
and then, $H_{f,g}\left( 0^{+}\right) =0$, $H_{f,g}\left( \pi /2^{-}\right)
=1-\left( 2/\pi \right) ^{p}>0$. Employing part (ii) of Theorem \ref%
{T-Pinelis4-Yb}, we conclude that there is a $x_{0}\in \left( 0,\pi
/2\right) $ such that $f/g$ is increasing on $\left( 0,x_{0}\right) $ and
decreasing on $\left( x_{0},\pi /2\right) $. It is derived that%
\begin{eqnarray*}
\frac{1}{3} &=&\lim_{x\rightarrow 0^{+}}\frac{f\left( x\right) }{g\left(
x\right) }<\frac{f\left( x\right) }{g\left( x\right) }<\frac{f\left(
x_{0}\right) }{g\left( x_{0}\right) }\text{ for }x\in \left( 0,x_{0}\right) ,
\\
1-\left( \frac{2}{\pi }\right) ^{p} &=&\lim_{x\rightarrow \pi /2^{-}}\frac{%
f\left( x\right) }{g\left( x\right) }<\frac{f\left( x\right) }{g\left(
x\right) }<\frac{f\left( x_{0}\right) }{g\left( x_{0}\right) }\text{ for }%
x\in \left( x_{0},\pi /2\right) ,
\end{eqnarray*}%
which leads us to%
\begin{equation*}
\theta _{1}\left( p\right) =\min_{p\in \left( 4/5,1\right) }\left( 1-\left( 
\frac{2}{\pi }\right) ^{p},\frac{1}{3}\right) <\frac{\left( \frac{\sin x}{x}%
\right) ^{p}-1}{\left( \cos x\right) ^{p}-1}\leq \frac{\left( \frac{\sin
x_{0}}{x_{0}}\right) ^{p}-1}{\left( \cos x_{0}\right) ^{p}-1}=\theta
_{0}\left( p\right)
\end{equation*}%
for all $x\in \left( 0,1\right) $, that is, the double inequality (\ref{Z-Y1}%
) holds true.

When $p=p_{0}=\left( \ln 3-\ln 2\right) /\left( \ln \pi -\ln 2\right) $, $%
\theta _{1}\left( p\right) =1/3$, which yields (\ref{Z-Y2}). Solving the
equation (\ref{H=0-t}) by mathematical software, we find that $x_{0}\approx
0.0000003658089313760$ and $\theta _{0}\left( p_{0}\right) \approx 0.33334$.

This completes the proof of this proposition.
\end{proof}

\end{document}